\documentclass[11pt,a4paper]{amsart}

\title[Integral representation of moderate cohomology]{Integral representation of moderate cohomology}

\author{H{\aa}kan Samuelsson Kalm}
\thanks{Author partially supported by the Swedish Research Council}
\subjclass[2010]{32A26, 32A27, 32C25, 32C35}
\keywords{moderate cohomology, integral representation, residue current, coherent ideal sheaf, complex space}
\address{H{\aa}kan Samuelsson Kalm, Department of Mathematical Sciences, Division of Algebra and Geometry, University of Gothenburg and 
Chalmers University of Technology, SE-412 96 G\"{o}teborg, Sweden}
\email{hasam@chalmers.se}

\date{\today}

\usepackage[english]{babel}
\usepackage{amsmath,calligra}
\usepackage{amsthm,amssymb,latexsym}
\usepackage[all,cmtip]{xy}
\usepackage{url,ae}
\usepackage{t1enc}
\usepackage{mathrsfs}
\usepackage{graphicx}
\usepackage{geometry}
\newtheorem{proposition}{Proposition}[section]
\newtheorem{theorem}[proposition]{Theorem}
\newtheorem{lemma}[proposition]{Lemma}

\theoremstyle{definition}
\newtheorem{definition}[proposition]{Definition}
\newtheorem{example}[proposition]{Example}
\newtheorem{remark}[proposition]{Remark}

\numberwithin{equation}{section}

\DeclareMathOperator{\Hom}{\mathscr{H}\text{\kern -3pt {\calligra\Large om}}\,}
\DeclareMathOperator{\Ext}{\mathscr{E}\text{\kern -3pt {\calligra\Large xt}}\,\,}
\DeclareMathOperator{\Image}{\mathscr{I}\text{\kern -3pt {\calligra\Large m}}\,}
\DeclareMathOperator{\Ker}{\mathscr{K}\text{\kern -3pt {\calligra\Large er}}\,}
\newcommand{\C}{\mathbb{C}}
\newcommand{\debar}{\bar{\partial}}

\newcommand{\HH}{\mathscr{H}}
\newcommand{\R}{\mathbb{R}}
\newcommand{\J}{\mathcal{J}}

\newcommand{\PM}{\mathscr{P} \kern -3pt \mathscr{M}}

\newcommand{\hol}{\mathscr{O}}

\newcommand{\K}{\mathscr{K}}

\newcommand{\Proj}{\mathscr{P}}

\newcommand{\CH}{\mathscr{C} \kern -2pt \mathscr{H}}
\newcommand{\B}{\mathbb{B}}
\newcommand{\Om}{\mathit{\Omega}}
\newcommand{\ett}{\mathbf{1}}

\def\newop#1{\expandafter\def\csname #1\endcsname{\mathop{\rm #1}\nolimits}}
\newop{span}

\begin{document}
\nocite{*}
\bibliographystyle{plain}

\begin{abstract}
We make the classical Dickenstein-Sessa canonical representation in local moderate cohomology
explicit by an integral formula. We also provide a similar representation of the higher local moderate cohomology groups.
The results are related to holomorphic forms 
on non-reduced complex spaces.


\end{abstract}

\maketitle
\thispagestyle{empty}

\section{Introduction}
Let $M$ be a complex $N$-dimensional manifold and $Y\subset M$ an analytic subset of pure codimension $\kappa$. 
If $\mu$ is a $\debar$-closed $(p,\kappa)$-current on $M$ with support in $Y$, then by the classical result of
Dickenstein and Sessa in \cite{DS2} and \cite{DS} there is locally a decomposition
\begin{equation}\label{grund}
\mu=\nu+\debar\tau,
\end{equation}
where $\nu$ is a uniquely determined so called Coleff-Herrera current with support in $Y$ and $\tau$ is a current with support in $Y$.
Thus, the local moderate cohomology associated with $Y$ is canonically represented by the Coleff-Herrera currents supported in $Y$.
In addition, if $\mathcal{J}\subset\hol_M$ is an ideal sheaf with $Z(\J)=Y$ such that $\mathcal{J}\mu=0$, i.e.,
$h\mu=0$ for any $h\in \mathcal{J}$, then $\J\nu=0$ and $\tau$ can be chosen so that $\J\tau=0$, see, e.g., \cite{BjorkAbel} or \cite{MatsNoether}.
The known proofs of the Dickenstein-Sessa decomposition do not give $\nu$ in \eqref{grund} explicitly.
The main objective of this paper is to find $\nu$ in \eqref{grund} in terms of $\mu$ by an explicitly integral formula.

Coleff-Herrera currents are modeled on the Coleff-Herrera product $\nu_{f}:=\debar(1/f_1)\wedge\cdots\wedge\debar(1/f_{\kappa})$
associated to a locally complete intersection ideal $\mathcal{J}_f=\langle f_1,\ldots,f_k\rangle\subset \hol_M$ introduced in \cite{CH}.
If $\J_f\mu=0$ then $\nu$ in \eqref{grund}
is of the form $\xi\wedge\nu_{f}$ for some holomorphic $p$-form $\xi$, see \cite{DS}. 

We are interested in a general coherent ideal sheaf $\J\subset\hol_M$ with $Y=Z(\mathcal{J})$ and we
follow the approach of Bj\"{o}rk to Coleff-Herrera currents with support in $Y$, see, e.g., \cite[Section~6.2]{BjorkAbel}.
Let $\mathcal{J}_Y\subset \hol_M$ be the ideal sheaf of functions vanishing on $Y$. 
The sheaf of Coleff-Herrera currents of degree $p$ with support in $Y$, $\CH_Y^p$, is 
the sheaf of germs of $\debar$-closed $(p,\kappa)$-currents $\nu$ such that $\overline{\mathcal{J}}_Y\nu=0$ and
with the \emph{standard extension property} (SEP) with respect to $Y$. 
That a current has the SEP with respect to $Y$ means roughly speaking that it has no mass concentrated on
proper analytic subsets of $Y$, see Section~\ref{PMsektion} below, and the condition $\overline{\J}_Y\nu=0$ means that $\nu$ only involves
holomorphic derivatives. If $Z$ is a cycle with $|Z|=Y$ then the integration current $[Z]$ is a section of $\CH_Y^{\kappa}$.

Our considerations are local or semi-global so from now on $M\subset \C^N$ is a pseudoconvex domain. 
Given a Hermitian free resolution of $\hol_M/\J$ (which always exists in relatively compact open subsets of $M$),
in \cite{AW1} Andersson and Wulcan introduced an associated vector-valued current 
$R=R_{\kappa}+R_{\kappa+1}+\cdots$,
where $R_k$ is a $(0,k)$-current taking values in an auxiliary trivial vector bundle $E_k$, 
such that a holomorphic function $h$ is a section of $\J$ (locally)
if and only if $hR=0$. If $\J$ is a complete intersection then $R$ is the corresponding Coleff-Herrera product. 
In \cite{MatsNoether} Andersson shows that any current in $\CH_Y^p$ is of the form $\xi\cdot R_{\kappa}$ for some
holomorphic $p$-form $\xi$ with values in $E_k^*$.

\begin{theorem}\label{ny}
Given $\mathcal{J}\subset\hol_M$, $M'\Subset M$, and $p\geq 0$ there is an integral kernel $P(\zeta,z)$ such that 
$\zeta\mapsto P(\zeta,z)$ is a holomorphic $p$-form in $M'$ with values in $E_{\kappa}^*$,
$z\mapsto P(\zeta,z)$ is a smooth
compactly supported $(N-p,N-\kappa)$-form in $M$, and if $\mu$ is a $\debar$-closed 
$(p,\kappa)$-current in $M$ with $\mathcal{J}\mu=0$ then,
in $M'$, $\nu$ in \eqref{grund} is given by
\begin{equation*}
\check{\Proj}\mu := R_{\kappa}\cdot \int_z P(\zeta,z)\wedge\mu(z).
\end{equation*}
\end{theorem}

The integral means the action of $\mu$ on the test form $z\mapsto P(\zeta,z)$.
The kernel $P(\zeta,z)$ is explicitly constructed given a free resolution of $\hol_M/\mathcal{J}$, 
see Section~\ref{intopsektion} below.
The operator $\check{\Proj}$ maps any $(p,\kappa)$-current in $M$ to a $(p,\kappa)$-current in $M'$ annihilated by both $\J$
and $\overline{\sqrt{\J}}$ and
with the SEP with respect to $Y$. 

One application of Theorem~\ref{ny} is to the problem of factorizing cycles. Recall that if
$\mathcal{J}=\langle f_1,\ldots,f_{\kappa}\rangle$ is a complete intersection and $Z$ is the corresponding fundamental cycle,
i.e., $Z=\sum_j m_jY_j$ where $m_j$ are certain multiplicities and $Y_j$ are the irreducible components of $Z(\J)$,
then 
\begin{equation}\label{PLCH}
[Z]=\frac{1}{(2\pi i)^{\kappa}} \debar\frac{1}{f_1}\wedge\cdots\wedge\debar\frac{1}{f_{\kappa}}\wedge df_{\kappa}\wedge\cdots\wedge df_1,
\end{equation}
see \cite{CH}. This was globalized to locally complete intersections by Demailly and Passare, \cite{DP}, and further generalized 
by Andersson in \cite{MatsLelong}. Recently, L{\"a}rk{\"a}ng and Wulcan, \cite{LW}, proved a formula similar to \eqref{PLCH}
for the fundamental cycle of a quite general complex subspace. By Theorem~\ref{ny} we get that if $Z$ is any cycle
with $|Z|=Z(\J)$ then
\begin{equation*}
[Z]=R_{\kappa}\cdot \int_z P(\zeta,z)\wedge [Z].
\end{equation*}




\smallskip

Let $\mathscr{C}_{\J}^{p,q}$ be the sheaf of 
$(p,q)$-currents in $M$ annihilated by $\J$ and notice that we have a complex $(\mathscr{C}_{\J}^{p,\bullet},\debar)$.
The Dickenstein-Sessa decomposition \eqref{grund} implies that there is a canonical isomorphism 
$\HH^{\kappa}(\mathscr{C}_{\J}^{p,\bullet},\debar)\simeq \CH^p_{\J}$, where $\CH^p_{\J}$ is the subsheaf of
$\CH^p_{Y}$ of currents annihilated by $\J$. The map $\CH^p_{\J}\to \HH^{\kappa}(\mathscr{C}_{\J}^{p,\bullet},\debar)$
is induced by the inclusion $\CH^p_{\J} \hookrightarrow \mathscr{C}_{\J}^{p,\kappa}$ and the inverse map
$\HH^{\kappa}(\mathscr{C}_{\J}^{p,\bullet},\debar)\to\CH^p_{\J}$ is induced by our operator $\check{\Proj}$.
The second objective of this paper is to give a similar description of the higher cohomology
$\HH^{q}(\mathscr{C}_{\J}^{p,\bullet},\debar)$, $q>\kappa$. It is well-known that
$\HH^{q}(\mathscr{C}_{\J}^{p,\bullet},\debar)=0$ for $q<\kappa$, see, e.g., \cite{BjorkDmod}.

\begin{theorem}\label{egenDSthm1}
There are fine sheaves of currents $\mathscr{B}_{\J}^{p,q}\subset \mathscr{C}_{\J}^{p,q}$, $p\geq 0$, $q\geq\kappa$,
such that
$\debar$ maps $\mathscr{B}_{\J}^{p,q}$ to $\mathscr{B}_{\J}^{p,q+1}$ and
the inclusion $\mathscr{B}_{\J}^{p,q}\subset \mathscr{C}_{\J}^{p,q}$ induces an isomorphism
$\HH^q\big(\mathscr{B}_{\J}^{p,\bullet},\debar\big)\simeq \HH^q\big(\mathscr{C}_{\J}^{p,\bullet},\debar\big)$.
Setting $\mathscr{B}_{\J}^{p,q}=0$ for $q<\kappa$ this holds for all $q$.
Moreover, $\overline{\sqrt{\J}}\mathscr{B}_{\J}^{p,q}=d\overline{\sqrt{\J}}\wedge\mathscr{B}_{\J}^{p,q}=0$.
\end{theorem}

The point is that $\mathscr{B}_{\J}^{p,q}$ is much smaller and less singular than $\mathscr{C}_{\J}^{p,q}$. 
Moreover, $\mathscr{B}_{\J}^{p,q}$ is a concretely defined subsheaf of the sheaf of
\emph{pseudomeromorphic currents}, see below, and the kernel of $\debar$ in 
$\mathscr{B}_{\J}^{p,\kappa}$ is $\CH_{\J}^p$.

Theorem~\ref{egenDSthm1} shows that the natural inclusion of complexes 
$(\mathscr{B}_{\J}^{p,\bullet},\debar) \hookrightarrow (\mathscr{C}_{\J}^{p,\bullet},\debar)$ is a 
quasi-isomorphism, i.e., an isomorphism on cohomology. 
Our final result provides in particular an explicit projection operator 
$\check{\Proj}\colon (\mathscr{C}_{\J}^{p,\bullet},\debar) \to (\mathscr{B}_{\J}^{p,\bullet},\debar)$ giving the inverse
of this quasi-isomorphism. Let $\mathscr{C}^{p,q}$ be the sheaf of $(p,q)$-currents in $M$.  

\begin{theorem}\label{egenDSthm2}
There is an integral operator 
$\check{\Proj}\colon \mathscr{C}^{p,q}(M)\to \mathscr{B}_{\J}^{p,q}(M')$ 
giving a quasi-isomorphism of complexes $(\mathscr{C}_{\J}^{p,\bullet},\debar) \to (\mathscr{B}_{\J}^{p,\bullet},\debar)$.
Moreover, there is an integral operator
$\check{\K}\colon \mathscr{B}_{\J}^{p,q}(M)\to \mathscr{B}_{\J}^{p,q-1}(M')$ such that
$\mu_{\restriction_{M'}} = \debar\check{\K}\mu + \check{\K}\debar\mu + \check{\Proj}\mu$
for any $\mu\in \mathscr{B}_{\J}^{p,q}(M)$.
\end{theorem}

\begin{remark}
It is well-known that the cohomology sheaves $\HH^q(\mathscr{C}_{\J}^{p,\bullet},\debar)$ are isomorphic to 
$\Ext^{q}(\hol_M/\J,\Om^p_M)$, where $\Om^p_M$ is the sheaf of holomorphic $p$-forms on $M$,
cf.\ Section~\ref{Bsektion} below.
\end{remark}

Assume temporarily that $Y$ is smooth and that $\mathcal{J}=\J_Y$; let $n=N-\kappa=\text{dim}\, Y$. Then 
in view of Example~\ref{exempel} below,
$\mathscr{B}_{\J}^{N,q}=i_*\mathscr{E}_Y^{n,q-\kappa}$, where $\mathscr{E}_Y$ is the sheaf of smooth forms on $Y$
and $i\colon Y\to M$ is the inclusion. Moreover, it is well-known that
$\CH_{\J}^N=i_*\Om_Y^{n}$,
see, e.g., \cite{AS}. Thus,
\begin{equation}\label{Bkomplex}
0\to \mathscr{B}_{\J}^{N,\kappa} \stackrel{\debar}{\longrightarrow} \cdots \stackrel{\debar}{\longrightarrow}
\mathscr{B}_{\J}^{N,N} \to 0
\end{equation}
is $i_*$ of the Dolbeault complex of smooth $(n,\bullet)$-forms on $Y$ and the cohomology at $\mathscr{B}_{\J}^{N,\kappa}$ 
is $i_*$ of the holomorphic $n$-forms on $Y$.
For not necessarily smooth $Y$ and $\J\subset\J_Y$ of pure codimension $\kappa$, Andersson and L\"{a}rk\"{a}ng
introduce a notion of holomorphic top-degree forms on the possibly non-reduced complex space $Y_{\J}=(Y,\hol_M/\J_{\restriction_Y})$,
see \cite[Section~5]{AL}; see \cite{Barlet} for the reduced case. Via $i_*$ this notion precisely corresponds to $\CH_{\J}^N$.  
By analogy it is reasonable to think of \eqref{Bkomplex} as a certain Dolbeault complex for $Y_{\J}$
with the cohomology at $\mathscr{B}_{\J}^{N,\kappa}$ being the holomorphic top-degree forms on $Y_{\J}$
and $\check{\Proj}$ as a projection operator onto these forms.

\bigskip

{\bf Acknowledgment:} I would like to thank Mats Andersson and Elizabeth Wulcan for stimulating discussions
on the topic of this paper.

\section{Pseudomeromorphic currents and weighted integral formulas}
\subsection{Pseudomeromorphic currents}\label{PMsektion}
In one complex variable $z$ it is elementary to see that the principal value current $1/z^m$ exists
and can be defined, e.g., as the limit as $\epsilon \to 0$ in the sense of currents
of $\chi(|h(z)|^2/\epsilon)/z^m$, where $h$ is a non-trivial holomorphic function (or tuple) vanishing at $z=0$
and $\chi$ is a smooth regularization of the characteristic function
of $[1,\infty)\subset \R$; for the rest of the paper $\chi$ will denote such a function. 
The current $1/z^m$ can also be defined as the value at $\lambda=0$ of the analytic continuation
of the current-valued function $\lambda \mapsto |h(z)|^{{2\lambda}}/z^m$. 
It follows that the \emph{residue current} $\debar(1/z^m)$ can be computed as the 
limit of $\debar\chi(|h(z)|^2/\epsilon)/z^m$ or as the value at $\lambda=0$ of 
$\lambda \mapsto \debar |h(z)|^{{2\lambda}}/z^m$.
Since tensor products of currents are well-defined we can form the current
\begin{equation}\label{elementary}
\tau=\debar \frac{1}{z_1^{m_1}}\wedge \cdots \wedge \debar \frac{1}{z_r^{m_r}}\wedge 
\frac{\gamma(z)}{z_{r+1}^{m_{r+1}}\cdots z_M^{m_M}}
\end{equation}
in $\C^M$, where $m_1,\ldots,m_r$ are positive integers, $m_{r+1},\ldots,m_{M}$ are nonnegative integers, and 
$\gamma$ is a smooth compactly supported form. Notice that $\tau$ is anti-commuting in the residue factors 
$\debar(1/z_j^{m_j})$ and commuting in the principal value factors $1/z_k^{m_k}$.
We say that a current of the form \eqref{elementary} is an \emph{elementary pseudomeromorphic current}. 
A current $\mu$ on a complex manifold $X$ is \emph{pseudomeromorphic} if and only if $\mu$ is a locally
finite sum of currents of the form $\pi_*\tau$, where $\tau$ is of the form \eqref{elementary} and $\pi$ is a holomorphic
map from a neighborhood of $\text{supp}\, \gamma$ to $X$, see \cite[Theorem~2.15]{AW3}. 
Currents on reduced complex spaces are also defined, see \cite{HL}.
A current $\mu$ on a reduced pure-dimensional complex space $X$ is
pseudomeromorphic if and only if there is a modification $\pi\colon X'\to X$ with $X'$ smooth and a pseudomeromorphic current
$\mu'$ on $X'$ such that $\mu=\pi_*\mu'$, \cite[Theorem~2.15]{AW3}. This yields the subsheaf $\PM_X$ of the sheaf of 
germs of currents on any reduced pure-dimensional complex space $X$. Notice that, since $\debar$ maps an elementary
pseudomeromorphic current to a sum of such currents, $\debar$ maps $\PM$ to itself.
Moreover, if $X$ and $Z$ are reduced pure-dimensional complex spaces and $\mu\in\PM(X)$, then $\mu\otimes 1\in \PM(X\times Z)$,
see \cite[Section~2]{AS}. Below, we will omit ``$\otimes 1$'' and write, e.g., $\mu(x)$ to denote on what variables a current
depends.  

\medskip

\noindent {\bf Dimension principle.}(\cite[Corollary~2.4]{AW2}, \cite[Proposition~2.3]{AS}) 
\emph{Let $X$ be a reduced pure-dimensional
complex space, let $\mu\in\PM(X)$, and assume that $\mu$ has support contained in a subvariety $V\subset X$. 
If $h\in\hol_X$ vanishes on $V$ then $\bar{h}\mu=d\bar{h}\wedge\mu=0$.
Moreover, if $\mu$ has bidegree $(*,q)$ and $\textrm{codim}_X V>q$, then $\mu=0$.}

\smallskip

For an analytic subset $Y\subset X$ of pure codimension $\kappa$, the sheaf $\CH_Y^p$ is characterized as the 
subsheaf of $\PM^{p,\kappa}_X$ of germs of $\debar$-closed currents with support in $Y$, see \cite{MatsUnique}. 

Pseudomeromorphic currents can be ``restricted''
to analytic (or constructible) subsets: Let $\mu\in\PM(X)$, let $V\subset X$ be an analytic subset,
and set $V^c:=X\setminus V$. Then the restriction of $\mu$ to the open subset $V^c$ has a natural 
pseudomeromorphic extension $\ett_{V^c}\mu$ to $X$. In \cite{AW2}, $\ett_{V^c}\mu$ is defined as the 
value at $0$ of the analytic continuation of the current-valued function $\lambda\mapsto |h|^{2\lambda}\mu$,
where $h$ is any holomorphic tuple with zero set $V$; $\ett_{V^c}\mu$ can also be defined as
$\lim_{\epsilon\to 0}\chi(|h|^2v/\epsilon)\mu$, where $v$ is any smooth strictly positive function,
see \cite[Lemma~2.6]{AW3}, cf.\ also \cite[Lemma~6]{LSprodukter}. 
The current $\ett_V\mu:=\mu-\ett_{V^c}\mu$ thus
is a pseudomeromorphic current with support contained in $V$, and if $\text{supp}\,\mu\subset V$, then $\ett_V\mu=\mu$. 
Moreover, see \cite[Section~2.2]{AW3}, if $V$ and $W$ are any 
constructible subsets then $\ett_V\ett_W\mu=\ett_{V\cap W}\mu$.
A current $\mu\in \PM(X)$ has the standard extension property (SEP) with respect to an 
analytic subsets $V\subset X$ if $\ett_W\mu=0$ for all germs of analytic subsets $W\subset X$ such that $\text{codim}_V W\cap V>0$.

Recall that a current on $X$ is said to be semi-meromorphic if it is a principal value current of the form
$\alpha/f$, where $\alpha$ is a smooth form and $f$ is a holomorphic function or section 
of a line bundle such that $f$ does not vanish identically on any component of $X$.
Following \cite{AS}, see also \cite[Section~4]{AW3}, we say that a current $a$ on $X$ is 
\emph{almost semi-meromorphic} if there is a modification $\pi\colon X'\to X$ and a semi-meromorphic 
current $\alpha/f$ on $X'$ such that $a=\pi_*(\alpha/f)$; if $f$ takes values in $L\to X'$ we need
also $\alpha$ to take values in $L\to X'$. 
If $a$ is almost semi-meromorphic on $X$, then the smallest Zariski-closed set outside of which $a$ is smooth
has positive codimension and is denoted $ZSS(a)$, the \emph{Zariski-singular support} of $a$, see \cite[Section~4]{AW3}.

For proofs of the statements in this paragraph we refer to \cite[Section~4]{AW3}, see also \cite[Section~2]{AS}.
Let $a$ be an almost semi-meromorphic current on $X$ and let $\mu\in\PM(X)$. Then there is a 
unique pseudomeromorphic current $T$ on $X$ coinciding with $a\wedge\mu$ outside of $ZSS(a)$ and 
such that $\ett_{ZSS(a)}T=0$. If $h$ is a holomorphic tuple, or section of a 
Hermitian vector bundle, such that $\{h=0\}=ZSS(a)$, then $T=\lim_{\epsilon\to 0}\chi(|h|^2/\epsilon)a\wedge\mu$;
henceforth we will write $a\wedge\mu$ in place of $T$. One defines $\debar a\wedge \mu$ so that Leibniz' rule holds,
i.e., $\debar a\wedge \mu:=\debar(a\wedge\mu)-(-1)^{\textrm{deg}\, a}a\wedge\debar\mu$.
If $\mu$ is almost semi-meromorphic then $a\wedge \mu$ is almost semi-meromorphic and, in fact,
$a\wedge\mu=(-1)^{\textrm{deg}\, a \, \textrm{deg}\, \mu}\mu\wedge a$.

\subsection{Currents associated to generically exact complexes}\label{URsektion}
Let $M$ be an $N$-dimensional complex manifold and let $\J\subset\hol_M$ be a coherent ideal sheaf. 
Suppose that we have a complex
\begin{equation*}
0\to E_m\stackrel{f_m}{\longrightarrow} \cdots \stackrel{f_1}{\longrightarrow} E_0\to 0
\end{equation*}
of holomorphic Hermitian vector bundles, with $E_0$ being the trivial line bundle, 
such that the associated sheaf complex $(\hol(E_{\bullet}),f_{\bullet})$
is a resolution of $\hol_M/\J$. The bundle $E:=\oplus_jE_j$ gets a natural superstructure by setting $E^+:=\oplus_jE_{2j}$
and $E^-:=\oplus_jE_{2j+1}$.
Following \cite{AW1} we define pseudomeromorphic currents $U$ and $R$ with values in $\textrm{End}(E)$ associated with the
Hermitian complex $(E_{\bullet},f_{\bullet})$. Notice that $\textrm{End}(E)$
gets an induced superstructure and so spaces of forms and currents with values in $E$ or $\textrm{End}(E)$
get superstructures as well.  
Let $f:=\oplus_jf_j$ and set 
$\nabla:=f-\debar$, which then becomes an odd mapping on spaces of forms or currents with values in $E$
such that $\nabla^2=0$;
notice that $\nabla$ induces an odd mapping $\nabla_{\textrm{End}}$ on $\textrm{End}(E)$-valued forms or currents
such that $\nabla^2_{\textrm{End}}=0$.
Outside of $Y=Z(\J)$, $(E_{\bullet},f_{\bullet})$ is pointwise exact and we 
let $\sigma_k\colon E_{k-1}\to E_k$ be the pointwise minimal inverse of $f_k$.
Set $\sigma:=\sigma_1+\sigma_2+\cdots$ and let $u:=\sigma+\sigma\debar\sigma + \sigma(\debar\sigma)^2+\cdots$.
Notice that $u=\sum_{0\leq \ell <k}u^{\ell}_k$,
where $u^{\ell}_k:=\sigma_{k}\debar\sigma_{k-1}\cdots\debar\sigma_{\ell+1}$ is a smooth 
$\textrm{Hom}(E_{\ell},E_{k})$-valued $(0,k-\ell-1)$-form outside of $Y$.
One can show that $\nabla_{\textrm{End}}u=\textrm{Id}_{E}$.
The form $u$ can be exended as a current across $Z$ by setting
\begin{equation}\label{Ureg}
U:=\lim_{\epsilon\to 0}\chi(|F|^2/\epsilon)u,
\end{equation}
where $F$ is a (non-trivial) holomorphic tuple vanishing on $Y$, see, e.g., \cite[Section~2]{AW1}.
As with $u$ we will write $U=\sum_{0\leq \ell<k}U^k_{\ell}$, where now $U^{\ell}_k$ is a 
$\textrm{Hom}(E_{\ell},E_{k})$-valued $(0,k-\ell-1)$-current; in fact, $U$ is an almost semi-meromorphic current, cf., e.g., \cite{AW3}.
The current $R$ is defined by $\nabla_{\textrm{End}} U=\textrm{Id}_{E} - R$
and hence $R$ is supported on $Y$ and $fR-\debar R=\nabla_{\textrm{End}}R=0$. 
Notice that $R$ is an almost semi-meromorphic current plus $\debar$ of such a current.
One can check that 
\begin{equation}\label{Rreg}
R=\lim_{\epsilon\to 0}\, \big(1-\chi(|F|^2/\epsilon)\big)\textrm{Id}_{E} + \debar\chi(|F|^2/\epsilon)\wedge u.
\end{equation}
We write $R=\sum_{0\leq\ell<k}R^{\ell}_k$, where $R^{\ell}_k$ is a 
$\textrm{Hom}(E_{\ell},E_{k})$-valued $(0,k-\ell)$-current. 
Since $E_0$ is the trivial line bundle we have
$\text{Hom}(E_0,E_k)\simeq E_k$ and we may identify $R^0_k$ with an $E_k$-valued current; 
sometimes we just write $R_k$ for $R^0_k$.

Since the sheaf complex $(\hol(E_{\bullet}),f_{\bullet})$ is supposed to be a resolution of $\hol_M/\J$, it follows
from \cite{AW1} that $R=R_{\kappa}^0+R_{\kappa+1}^0+\cdots$, where $\kappa=\text{codim}\, Y$,
and that a holomorphic function $g$ is a section of $\J$ if and only if the $E$-valued current $Rg$ vanishes. 

\subsection{Weighted integral formulas}\label{intformelsektion}
We apply Andersson's method, \cite{MatsIntrep}, of generating weighted integral formulas in a domain
$D\subset \C^N$. To begin with, suppose that $k(\zeta,z)$ is an integrable
$(N,N-1)$-form in $D\times D$ and $p(\zeta,z)$ is a smooth $(N,N)$-form in $D\times D$ such that
\begin{equation}\label{currentKoppel}
\debar k(\zeta,z) = [\Delta^D] - p(\zeta,z),
\end{equation}
where $[\Delta^D]$ is the current of integration along the diagonal. Applying \eqref{currentKoppel}
to test forms of the form $\psi_{\epsilon}(z)\wedge \varphi(\zeta)$, where $\psi_{\epsilon}$ is an approximate identity
and $\varphi$ is a test form in $D$,
one obtains the integral formula
\begin{equation}\label{penna}
\varphi(z) = \debar_{z} \int_{D_{\zeta}} k(\zeta,z)\wedge \varphi(\zeta) + 
\int_{D_{\zeta}} k(\zeta,z)\wedge \debar\varphi(\zeta) + \int_{D_{\zeta}} p(\zeta,z)\wedge \varphi(\zeta)
\end{equation}
for all $z\in D$ by letting $\epsilon\to 0$.

Following \cite{MatsIntrep}, to find such $k$ and $p$ 
let $\eta=(\eta_1,\ldots,\eta_N)$ be a holomorphic tuple in $D\times D$ that defines the diagonal and let 
$\Lambda_{\eta}$ be the exterior algebra spanned by $\Lambda^{0,1}T^*(D \times D)$
and the $(1,0)$-forms $d\eta_1,\ldots,d\eta_N$. On forms with values in $\Lambda_{\eta}$ interior multiplication
with $2\pi i \sum\eta_j \partial/\partial \eta_j$, denoted $\delta_{\eta}$, is defined; set $\nabla_{\eta}=\delta_{\eta}-\debar$.
Let $s$ be a smooth $(1,0)$-form in $\Lambda_{\eta}$ such that $|s|\lesssim |\eta|$
and $|\eta|^2\lesssim |\delta_{\eta}s|$ and let $B=\sum_{k=1}^N s\wedge (\debar s)^{k-1}/(\delta_{\eta}s)^k$.
It is proved in \cite{MatsIntrep} that then $\nabla_{\eta} B = 1-[\Delta^D]$. Identifying terms of top degree
we see that $\debar B_{N,N-1} = [\Delta^D]$ and so \eqref{currentKoppel} is satisfied with $k(\zeta,z)=B_{N,N-1}$ and $p(\zeta,z)=0$.
For instance, if we take $s=\partial |\zeta-z|^2$ and $\eta=\zeta-z$, then the resulting $B$ is sometimes called
the full Bochner-Martinelli form and the term of top degree is the classical Bochner-Martinelli kernel.

Let $D_1, D_2\subset D$. A smooth section $g(\zeta,z)=g_{0}+\cdots +g_{N}$ of $\Lambda_{\eta}$
over $D_1\times D_2$ such that $\nabla_{\eta} g=0$ in $D_1\times D_2$ and $g_{0}(z,z)=1$ for $z\in D':=D_1\cap D_2$ 
is called a \emph{weight} 
with respect to $D_1\times D_2$; $g_{j}$ is the sum of the terms of $g$ of bidegree $(j,j)$. 
Notice that the exterior product of two weights again is a weight
(with respect to a suitable set).
If $g$ is a weight with respect to $D_1\times D_2$, then it follows that $\nabla_{\eta} (g\wedge B) = g-g\wedge[\Delta^D]$ 
in $D_1\times D_2$ and, identifying 
terms of bidegree $(N,N-1)$, we get that
\begin{equation}\label{gulp}
\debar (g\wedge B)_{N,N-1} = [\Delta^D] - g_{N}
\end{equation} 
in $D_1\times D_2$. It follows that if $\varphi$ is smooth
with compact support in $D_1$, then \eqref{penna}, with $k(\zeta,z)=(g(\zeta,z)\wedge B)_{N,N-1}$ and $p(\zeta,z)=g_{N}(\zeta,z)$, 
holds in $D_2$, and vice versa.

\begin{example}\label{hjalpvikt}
Let $D\Subset \C^N$ be pseudoconvex and let $K\subset D$ be a holomorphically convex compact subset. Let
$\rho$ be a smooth compactly supported function in $D$ that is $1$ in a neighborhood of $K$. 
One can find a smooth form $\tilde{s}(\zeta,z)=\tilde{s}_1(\zeta,z)d\eta_1 + \cdots + \tilde{s}_N(\zeta,z)d\eta_N$,
defined for $z$ in a neighborhood of $\text{supp}\,\debar\rho$ and $\zeta$ in a neighborhood $\tilde{D}$ of $K$,
such that $\zeta\mapsto \tilde{s}(\zeta,z)$ is holomorphic and $\delta_{\eta}\tilde{s}=1$. Then
\begin{equation*}
g(\zeta,z):=\rho(z) - \debar\rho(z)\wedge \sum_{k=1}^N\tilde{s}\wedge(\debar \tilde{s})^{k-1},
\end{equation*}
is a weight with respect to $\tilde{D}\times D$ that depends holomorphically on $\zeta$ and has
compact support in $D_z$; cf.\ \cite[Example~2]{MatsIntrepII} 
or \cite[Example~5.1]{AS} in case $D$ is the unit ball $\B$ in $\C^N$. 
\end{example}

\begin{example}\label{divvikt1}
Let $D\Subset \C^N$ be pseudoconvex, let $\J\subset\hol_D$ be a coherent sheaf of ideals in $D$, and assume that 
there is a free resolution $(\hol(E_{\bullet}),f_{\bullet})$ of $\hol_D/\J$ in $D$. Let $U=U(\zeta)$ and $R=R(\zeta)$ be associated currents,
let $U^{\epsilon}$ and $R^{\epsilon}$ be the regularizations given by \eqref{Ureg} and \eqref{Rreg}, respectively, 
and let $U^{\epsilon,\ell}_k$ and
$R^{\epsilon,\ell}_k$ be the parts taking values in 
$\text{Hom}(E_{\ell},E_k)$.
By \cite[Proposition~5.3]{MatsIntrepII}
we can, for $\ell\leq k$, find \emph{Hefer morphisms} 
$H^{\ell}_k$, which are holomorphic sections of $\Lambda_{\eta}\otimes \text{Hom}(E^{\zeta},E^z)$,
such that $H^{\ell}_k$ is a holomorphic $k-\ell$-form with values in $\text{Hom}(E^{\zeta}_k,E^z_{\ell})$
and 
\begin{equation*}
H_k^{k}\restriction_{\Delta_D} = \textrm{Id}_{E_k} \quad \textrm{and} \quad
\delta_{\eta} H_k^{\ell} = H_{k-1}^{\ell}f_k - f_{\ell+1}(z)H_k^{\ell+1}, \,\, k>\ell,
\end{equation*}
where $f_k=f_k(\zeta)$. One can check that then
\begin{equation*}
G^{\epsilon}:=
\sum_{k\geq 0} H^0_kR^{\epsilon,0}_k + f_1(z)\sum_{k\geq 1}H^1_{k}U^{\epsilon,0}_k,
\end{equation*}
is a weight with respect to $D\times D$ for all $\epsilon>0$; cf.\ \cite{MatsIntrepII}, \cite{AW1}, and \cite{AS}.
\end{example}


\section{Integral operators associated to an ideal sheaf}\label{intopsektion}
Let $D\Subset \C^N$ be pseudoconvex, let $\J\subset \hol_D$ be a coherent sheaf of ideals in $D$
with $Y=Z(\J)$ of pure codimension $\kappa$,
and let $D'\Subset D$. Let $g(\zeta,z)$ be any weight with respect to $D'\times D$ such that 
$z\mapsto g(\zeta,z)$ has compact support in some $D''\Subset D$ for all $\zeta$.
In $D''$ there is a free resolution
$(\hol(E_{\bullet}),f_{\bullet})$ of $\hol_D/\J$, associated currents $U=U(\zeta)$ and $R=R(\zeta)$, 
and, moreover, in $D''\times D''$ we can find associated Hefer morphisms $H^{\ell}_k$.
Then
$G^{\epsilon}=HR^{\epsilon}+f_1(z)HU^{\epsilon}$ of Example~\ref{divvikt1} is a weight with respect to $D''\times D''$.
It follows that $G^{\epsilon}\wedge g$ is a weight with respect to $D'\times D$ and has compact support in $D_z$.

Notice, in view of Example~\ref{hjalpvikt}, that we may choose the weight $g$ so that it contains no
$d\bar{\zeta}$-differentials and $\zeta\mapsto g(\zeta,z)$ is holomorphic; we then say that $g$ is holomorphic in $\zeta$.


\begin{lemma}\label{Koppel-test}
If $\varphi$ is a test form in $D'$, then for all $\epsilon>0$ and all $z\in D$
\begin{eqnarray*}
\varphi(z) &=& \debar_z \int_{\zeta} (HR^{\epsilon}\wedge g\wedge B)_{N,N-1}\wedge \varphi(\zeta) +
\int_{\zeta} (HR^{\epsilon}\wedge g\wedge B)_{N,N-1}\wedge \debar\varphi(\zeta) \\
& &
+\int_{\zeta} (HR^{\epsilon}\wedge g)_{N,N}\wedge\varphi(\zeta) + \phi \psi,
\end{eqnarray*}
where $\phi$ is a section of $\J$ and $\psi$ is some test form in $D$. The integrals on the right-hand side are test forms in $D$. 
\end{lemma}

\begin{proof}
Since $G^{\epsilon}\wedge g=HR^{\epsilon}\wedge g + f_1(z)HU^{\epsilon}\wedge g$ is a weight with respect
to $D'\times D$ and the entries of $f_1$ are sections of $\J$,  
it follows from Section~\ref{intformelsektion} that the claimed equality holds in $D$. 
Since $z\mapsto g(\zeta,z)$ has compact support in $D$, $\psi$ as well as the integrals on the right-hand side
are test forms in $D$.
\end{proof}

Let $\mu$ be an arbitrary current in $D$. Since $z$ and $\zeta-z$ are independent variables in $D\times D$,
the tensor product $B\wedge\mu(z)$ is well-defined, cf.\ \cite[Theorem~5.1.1]{Horm}. 
Let $\pi\colon D_{\zeta}\times D_z \to D_{\zeta}$ be the natural projection.
Then, since $z\mapsto g(\zeta,z)$ has compact support, 
\begin{equation*}
\pi_* \, (HR^{\epsilon}\wedge g \wedge B)_{N,N-1} \wedge \mu(z)  \quad \text{and} \quad
\pi_*\, (HR^{\epsilon}\wedge g)_{N,N} \wedge\mu(z)
\end{equation*}
are 
well-defined current in $D'_{\zeta}$. For notational convenience we will often write $\int_z \tau$ instead
of $\pi_*\tau$ for a current $\tau$ in $D\times D$.

\begin{lemma}\label{Koppel-curr}
Let $\mu$ be a current in $D$ such that $\J\mu=0$. Then, for all $\epsilon>0$,
\begin{eqnarray}\label{berk}
\mu(\zeta) &=& \int_z (HR^{\epsilon}\wedge g \wedge B)_{N,N-1} \wedge \debar\mu(z) +
\debar_{\zeta} \int_z (HR^{\epsilon}\wedge g \wedge B)_{N,N-1} \wedge \mu(z) \nonumber \\
& &
+ \int_z (HR^{\epsilon}\wedge g)_{N,N} \wedge\mu(z)
\end{eqnarray}
holds in $D'$.
\end{lemma}

\begin{proof}
Let $\varphi$ be a test form in $D'$. The action of the current 
$\int_z (HR^{\epsilon}\wedge g \wedge B)_{N,N-1} \wedge \debar\mu(z)$ on $\varphi(\zeta)$ is by definition the action 
of $(HR^{\epsilon}\wedge g \wedge B)_{N,N-1} \wedge \debar\mu(z)$ on $\varphi(\zeta)\otimes 1$ and this equals,
by \cite[Theorem~5.1.1.]{Horm}, the action of $\debar\mu(z)$ on $\int_{\zeta} (HR^{\epsilon}\wedge g \wedge B)_{N,N-1} \wedge \varphi(\zeta)$.
Hence,
\begin{equation*}
\left(\int_z (HR^{\epsilon}\wedge g \wedge B)_{N,N-1} \wedge \debar\mu(z)\right). \varphi(\zeta)=
\pm \mu(z). \left(\debar_z\int_{\zeta} (HR^{\epsilon}\wedge g \wedge B)_{N,N-1} \wedge \varphi(\zeta)\right).
\end{equation*}
Similar formulas hold for the second and third term on the right-hand side of \eqref{berk} and the lemma thus follows from
Lemma~\ref{Koppel-test}.
\end{proof}

\begin{definition}\label{Poperator}
For a current $\mu$ in $D$ we let $\check{\Proj}\mu$ be the current in $D'$ given by
\begin{equation*}
\check{\Proj}\mu (\zeta) = \int_z (HR\wedge g)_{N,N}\wedge \mu(z).
\end{equation*} 
\end{definition}
Notice, since $R=R(\zeta)$, that $(HR\wedge g)_{N,N}\wedge \mu(z)$ is well-defined as a tensor product in $D'\times D$.
Notice also that $\check{\Proj}$ maps arbitrary currents in $D$ to pseudomeromorphic 
currents in $D'$ annihilated by $\J$. 

The operator $\check{\Proj}$ of Theorems~\ref{ny} and \ref{egenDSthm2} corresponds to a choice of weight $g$ such that
$\zeta\mapsto g(\zeta,z)$ is holomorphic, but in the definition above we do not require this.

\begin{proposition}\label{fundprop}
Suppose that the weight $g$ in the definition of $\check{\Proj}$ depends holomorphically on $\zeta$.
Then, for any $(p,q)$-current $\mu$ in $D$, $\check{\Proj}\mu$ is of the form $\xi\cdot R_q^0$, where
\begin{equation*}
\xi=\Phi(\mu)=\int_z H^0_q\wedge g_{N-q}\wedge \mu(z)
\end{equation*}
is a holomorphic $p$-form with values in $E_q^*$.
Moreover, if $\J\mu=0$ then $\check{\Proj}\debar\mu=\debar\check{\Proj}\mu$ and $\Phi(\debar\mu)=f_{q+1}^*\Phi(\mu)$,
where $f_{q+1}^*$ is the transpose of the map $f_{q+1}\colon E_{q+1}\to E_q$.
\end{proposition}

For the rest of this section we will fix frames for the trivial $E_j$-bundles and associated dual frames for the $E^*_j$'s;
sections of the $E_j$'s will be represented by column vectors, sections of the $E^*_j$'s by row vectors, and maps between bundles
by matrices. Notice that $E_0$ is assumed to be the trivial line bundle so that $\text{Hom}(E_j,E_0)\simeq E_j^*$.
Recall that $f_k=f_k(\zeta)$.
To prove Proposition~\ref{fundprop} we will need

\begin{lemma}\label{keylemma} Letting $(\cdot)^*$ denote matrix transpose we have
\begin{equation*}
f^*_{q+1}(H^0_q)^*\wedge g_{N-q} = \debar \big((H^0_{q+1})^*\wedge g_{N-q-1}\big) + (f_1(z)H^1_{q+1})^*\wedge g_{N-q}.
\end{equation*}
\end{lemma}
\begin{proof}
This is verified by the following computation.
\begin{eqnarray*}
f^*_{q+1}(H^0_q)^*\wedge g_{N-q} &=& (H^0_qf_{q+1})^*\wedge g_{N-q} \\
&=&
(\delta_{\eta} H^0_{q+1} + f_1(z)H^1_{q+1})^*\wedge g_{N-q} \\
&=&
(\delta_{\eta}H^0_{q+1}\wedge g_{N-q})^* + (f_1(z)H^1_{q+1})^*\wedge g_{N-q} \\
&=&
(\pm H^0_{q+1}\wedge \delta_{\eta} g_{N-q})^* + (f_1(z)H^1_{q+1})^*\wedge g_{N-q} \\
&=&
(\pm H^0_{q+1}\wedge \debar g_{N-q-1})^* + (f_1(z)H^1_{q+1})^*\wedge g_{N-q} \\
&=&
\debar (H^0_{q+1}\wedge g_{N-q-1})^* + (f_1(z)H^1_{q+1})^*\wedge g_{N-q},
\end{eqnarray*}
where the second equality follows from the properties of the Hefer morphisms (see Example~\ref{divvikt1}),
the forth since $0=\delta_{\eta}(H^0_{q+1}\wedge g_{N-q})$
for degree reasons, the fifth since $g$ is a weight, and the sixth since $H^0_{q+1}$ is holomorphic.
\end{proof}

\begin{proof}[Proof of Proposition~\ref{fundprop}]
Since $g$ is holomorphic in $\zeta$ and in particular contains no $d\bar{\zeta}$-differentials,
it follows for degree reasons that 
\begin{equation*}
\check{\Proj}\mu(\zeta)=\int_zH^0_qR_q^0\wedge g_{N-q}\wedge\mu(z)
=\pm (R^{0}_q)^*\int_z(H^0_q)^*\wedge g_{N-q}\wedge\mu(z).
\end{equation*}
Since also $H$ is holomorphic in $\zeta$ the first statement of the proposition follows.
For the rest of this proof we will write $R_j$ instead of $R_j^0$.

Assume that $\J\mu=0$. To see that $\check{\Proj}\debar\mu=\debar\check{\Proj}\mu$ we compute:
\begin{eqnarray*}
\check{\Proj}\debar\mu &=& \int_zH^0_{q+1}R_{q+1}\wedge g_{N-q-1}\wedge\debar\mu(z) \\
&=&
\pm R^*_{q+1}\int_z \debar \big((H^0_{q+1})^*\wedge g_{N-q-1}\big)\wedge\mu(z) \\
&=&
\pm R^*_{q+1} \int_z f^*_{q+1}(H^0_q)^*\wedge g_{N-q}\wedge \mu(z)
\mp R^*_{q+1} \int_z (f_1(z)H^1_{q+1})^*\wedge g_{N-q}\wedge \mu(z) \\
&=&
\pm (f_{q+1}R_{q+1})^* \int_z (H^0_q)^*\wedge g_{N-q}\wedge \mu(z) \\
&=&
\debar \left(R^*_{q} \int_z (H^0_q)^*\wedge g_{N-q}\wedge \mu(z) \right) =\debar \check{\Proj}\mu,
\end{eqnarray*}
where the second equality holds since $R=R(\zeta)$ is independent of $z$, the third by Lemma~\ref{keylemma},
the forth since (the entries of) $f_1(z)$ annihilate $\mu(z)$, and the fifth since
$\nabla_{\text{End}} R=0$ (see Section~\ref{URsektion}) and $H$ and $g$ are holomorphic in $\zeta$.

Moreover, in view of Lemma~\ref{keylemma}, we get that
\begin{eqnarray*}
\Phi(\debar\mu)^* &=& \int_z \debar \big((H^0_{q+1})^*\wedge g_{N-q-1}\big)\wedge\mu(z) \\
&=& f_{q+1}^*\int_z (H^0_{q})^*\wedge g_{N-q}\wedge\mu(z) + 
\int_z\big(f_1(z)H^1_{q+1}\big)^*\wedge g_{N-q}\wedge\mu(z) \\
&=& f_{q+1}^*\Phi(\mu)^*
\end{eqnarray*}
since (the entries of) $f_1(z)$ annihilate $\mu(z)$.
\end{proof}

Let $\mu\in\PM(D)$. Since $B$ is almost semi-meromorphic in $D\times D$, the product $B\wedge \mu(z)$ 
is a well-defined pseudomeromorphic current in $D\times D$, in view of Section~\ref{PMsektion};
by the uniqueness in \cite[Theorem~5.1.1.]{Horm},
$B\wedge \mu(z)$ coincides with the tensor product of $B$ and $\mu$ ($z$ and $\zeta-z$ are independent variables).
From Section~\ref{URsektion}, $R$ is an almost semi-meromorphic current plus $\debar$ of such a current. 
Thus, by Section~\ref{PMsektion}, 
$R\wedge B\wedge \mu(z)$ is a well-defined
pseudomeromorphic current in $D'\times D$ that can be defined as
$\lim_{\epsilon\to 0}R^{\epsilon}\wedge B \wedge\mu(z)$, where $R^{\epsilon}$ is the regularization of $R$ given by \eqref{Rreg}.
Even though $R^0_j=0$ for $j<\kappa$
it may be the case that $\lim_{\epsilon\to 0}R^{0,\epsilon}_j\wedge B \wedge\mu(z)\neq 0$.
Still, the support of $R\wedge B\wedge\mu(z)$ is contained in $Y\times \text{supp}\,\mu$.
To see this it suffices, in view of \eqref{Rreg} and Section~\ref{PMsektion}, to see that $\ett_{Y\times D} B\wedge\mu(z)=0$.
Since $B$ is smooth outside of the diagonal $\Delta\subset D\times D$ it is clear that 
$\text{supp}\,\ett_{Y\times D} B\wedge\mu(z)\subset \Delta$. Moreover, $ZSS(B)\subset \Delta$ and so 
$\ett_{\Delta}B\wedge\mu(z)=0$. Hence,
\begin{equation*}
\ett_{Y\times D} B\wedge\mu(z)=\ett_{\Delta}\ett_{Y\times D} B\wedge\mu(z)
=\ett_{Y\times D} \ett_{\Delta}B\wedge\mu(z)=0.
\end{equation*}

\begin{definition}\label{Koperator}
For a pseudomeromorphic current $\mu$ in $D$ we let $\check{\K}\mu$ be the pseudomeromorphic current in $D'$ given by
\begin{equation*}
\check{\K}\mu (\zeta) = \int_z (HR\wedge g\wedge B)_{N,N-1}\wedge \mu(z).
\end{equation*} 
\end{definition}
As in the definition of the $\check{\Proj}$-operators, we do not require $g$ to be holomorphic in $\zeta$,
but the $\check{\K}$-operator of Theorem~\ref{egenDSthm2} corresponds to such a choice.

Notice that, since $\text{supp}\,R\wedge B \wedge\mu(z)\subset Y\times D$,
$\check{\K}$ maps pseudomeromorphic currents in $D$ to pseudomeromorphic currents in $D'$ with
support contained in $Y$. We do not know whether or not $\J\check{\K}\mu=0$ for a general pseudomeromorphic $\mu$.

The following proposition follows from Lemma~\ref{Koppel-curr} by letting $\epsilon\to 0$.

\begin{proposition}\label{DSweak}
For any $\mu\in\PM^{p,q}(D)$ such that $\J\mu=0$ 
we have $\mu_{\restriction_{D'}}=\debar\check{\K}\mu + \check{\K}\debar\mu + \check{\Proj}\mu$.
\end{proposition}

\begin{proposition}\label{objective1}
Let $\check{\Proj}$ be an operator associated to $\J$ and corresponding to a weight $g$ such that $\zeta\mapsto g(\zeta,z)$ is holomorphic.
If $\nu\in\CH^p_{\J}(D)$, then $\check{\Proj}\nu=\nu_{\restriction_{D'}}$. 
Moreover, if $\mu\in\mathscr{C}_{\J}^{p,\kappa}(D)$ is $\debar$-closed, 
then $\check{\Proj}\mu\in\CH_{\J}^p(D')$ and $\check{\Proj}\mu$ is the current $\nu$ in \eqref{grund}.
\end{proposition}

\begin{proof}
By Proposition~\ref{DSweak}, $\nu_{\restriction_{D'}}=\debar\check{\K}\nu + \check{\Proj}\nu$. However,
$\check{\K}\nu$ is a pseudomeromorphic $(p,\kappa-1)$-current with support contained in $Y$ and must 
thus vanish in view of the Dimension principle.

For the second statement, notice that $\check{\Proj}\mu$ is a pseudomeromorphic $(p,\kappa)$-current
annihilated by $\J$ (in particular with support in $Y$)
and, by Proposition~\ref{fundprop}, $\debar\check{\Proj}\mu=0$. Thus, in view of Section~\ref{PMsektion},
$\check{\Proj}\mu$ is a section of $\CH_{\J}^p$. Now consider the decomposition \eqref{grund} and recall that $\J\nu=0$ and 
that we may assume that $\J\tau=0$. By the first part of the proof and Proposition~\ref{fundprop}
\begin{equation*}
\check{\Proj}\mu= \check{\Proj}\nu + \check{\Proj}\debar\tau=
\nu + \debar\check{\Proj}\tau.
\end{equation*} 
Together with \eqref{grund} this gives that 
$\mu-\check{\Proj}\mu=\debar(\tau-\check{\Proj}\tau)$, and so by uniqueness of \eqref{grund}, $\check{\Proj}\mu=\nu$.
\end{proof}

\begin{proof}[Proof of Theorem~\ref{ny}]
Setting $P(\zeta,z):=H^0_{\kappa}\wedge g_{N-\kappa}$, where $g$ is a weight depending holomorphically on $\zeta$, 
Theorem~\ref{ny} follows from Propositions~\ref{fundprop}
and \ref{objective1}.
\end{proof}


\section{The sheaf $\mathscr{B}_{\mathcal{J}}^{p,q}$ and proof of Theorems~\ref{egenDSthm1} and \ref{egenDSthm2}}\label{Bsektion}
Let $M$ be an $N$-dimensional complex manifold and $\J\subset\hol_M$ a coherent ideal sheaf as before.
In analogy with \cite[Definition~7.1]{AS}, \cite[Definition~4.1]{RSWSerre}, and \cite[Section~6.2]{holoforms} we make the following definition. 
\begin{definition}
A $(p,q)$-current $\mu$ on $M$ is a section of $\mathscr{B}_{\J}^{p,q}$ over an open set $\mathcal{U}\subset M$ 
if for each $x\in \mathcal{U}$ there is a neighborhood $D$ of $x$ such that
$\mu_{\restriction_D}$ is a finite sum of currents of the form
\begin{equation}\label{kram}
\xi_m\wedge \check{\K}_m (\cdots \xi_1\wedge\check{\K}_1(\xi_0\cdot R) \cdots),
\end{equation}
where $R$ is a current corresponding to a Hermitian free resolution $(\hol(E_{\bullet}),f_{\bullet})$ of $\hol/\J$
in $D$,
$\xi_0$ is a smooth form with values in $\oplus_qE^*_q$, $\xi_j$ is a smooth form for $j\geq 1$, and the 
$\check{\K}_j$'s are operators as in Definition~\ref{Koperator}.
\end{definition}

It is clear from the definition that $\mathscr{B}_{\J}:=\oplus_{p,q}\mathscr{B}_{\J}^{p,q}$ is a module over
the sheaf of smooth forms and that it is closed under $\check{\K}$-operators.
Moreover, any section of $\mathscr{B}_{\J}$ has support in $Y$ and so, in view of 
Section~\ref{PMsektion}, if $\mu\in \mathscr{B}_{\J}^{p,\kappa}$ is $\debar$-closed then $\mu\in\CH_Y^p$.
Notice also, by Section~\ref{intopsektion}, that a $\check{\Proj}$-operator associated with $\J$ 
maps an arbitrary current to a section of $\mathscr{B}_{\J}$
and that $\mathscr{B}_{\J}^{p,q}$ is a subsheaf of $\PM^{p,q}$.

\begin{example}\label{exempel}
Suppose that $\J$ is the radical ideal sheaf of a complex submanifold $i\colon Y\hookrightarrow M$ of codimension $\kappa$.
Then $\mathscr{B}_{\J}^{N,q}=i_* \mathscr{E}_Y^{N-\kappa,q-\kappa}$. 
This may be verified locally so we assume that $M$ is the unit ball in $\C^N$ with coordinates $(z;w)$
such that $Y=\{z=0\}$ and $\J=\langle z\rangle$. For notational convenience we will also assume that $\kappa=1$. Let $\varphi$ be a smooth 
$(N-1,q-1)$-form on $Y$ and let $\tilde{\varphi}$ be any smooth extension of $\varphi$ to $M$. 
Then $i_*\varphi=\tilde{\varphi}\wedge [Y]$.
From the free resolution $0\to \hol \stackrel{z\cdot}{\longrightarrow} \hol$ of $\hol/\J$ we get the current $R=\debar(1/z)$
and so, by the Poincar\'e-Lelong formula, $\tilde{\varphi}\wedge [Y]=i\tilde{\varphi}\wedge dz\wedge R/2\pi$. Hence,
$i_* \mathscr{E}_Y^{N-\kappa,q-\kappa}\subset \mathscr{B}_{\J}^{N,q}$.

For the opposite inclusion, let $\mu$ be a section of $\mathscr{B}_{\J}^{N,q}$; we may assume that $\mu$ is of the form \eqref{kram}.
We use induction over $m$ in \eqref{kram} to see that $\mu=\xi\wedge [Y]$ for some smooth $\xi$;
notice that such a $\xi$ exists if there is a continuous $\xi'$, which is smooth along $Y$, such that $\mu=\xi'\wedge [Y]$.
If $(\hol(E_{\bullet}),f_{\bullet})$ is any free Hermitian resolution of $\hol/\J$ with $\text{rank}\, E_0=1$, then the associated 
residue current is of the form $\alpha \wedge\debar(1/z)$, where $\alpha$ is a smooth $(0,*)$-form with values in $\oplus_j E_j$, see
\cite[Theorem~4.4]{AW1}. Hence, if $\mu$ is of the form \eqref{kram} with $m=0$, then $\mu=\xi_0\cdot \alpha\wedge\debar (1/z)$,
where $\xi_0$ is a smooth $(N,*)$-form with values in $\oplus_jE^*_j$. Writing $\xi_0\cdot \alpha=\xi\wedge dz$, for some scalar-valued 
smooth form $\xi$, we get $\xi_0\cdot \alpha\wedge\debar (1/z)=-2\pi i \xi\wedge [Y]$, and the induction start follows. 
Now, if $\check{\K}$ is an integral operator as in
Definition~\ref{Koperator}, then 
\begin{equation*}
\check{\K}(\xi\wedge [Y]) = \int_{z\in Y} (HR\wedge g\wedge B)_{N,N-1}\wedge \xi(z)=R\cdot \int_{z\in Y} k(\zeta,z)\wedge \xi(z),
\end{equation*}
where $R=R(\zeta)$ is of the form $\alpha(\zeta)\wedge\debar (1/\zeta)$ as above, and $k(\zeta,z)$ is a kernel in $M\times M$
which is $\mathcal{O}(1/|\zeta-z|^{2(N-1)-1})$. 
It follows that $\zeta\mapsto \int_{z\in Y} k(\zeta,z)\wedge \xi(z)$ is a continuous $(N,*)$-form with values in 
$\oplus_jE^*_j$ and we write it as $\psi(\zeta)\wedge d\zeta$. Moreover, in view of \cite[Lemma~6.2]{AS}, $\psi_{\restriction_Y}$ is smooth.
We get
\begin{equation*}
\check{\K}(\xi\wedge [Y])=\alpha\wedge\debar (1/\zeta)\cdot \psi\wedge d\zeta=
\pm 2\pi i\alpha\cdot \psi\wedge [Y],
\end{equation*}
and the induction step follows. \hfill \qed
\end{example}

To prove Theorems~\ref{egenDSthm1} and \ref{egenDSthm2} we may assume that $M=D\Subset \C^N$ is a pseudoconvex domain.
Let $D'\Subset D$ and let $D''\Subset D$ be a suitable neighborhood of $\overline{D}'$ as in the beginning of 
Section~\ref{intopsektion}. 
Let 
\begin{equation*}
0\to \hol(E_m)\stackrel{f_m}{\longrightarrow} \cdots \stackrel{f_1}{\longrightarrow} \hol(E_0) \to \hol/\J\to 0
\end{equation*}
be a Hermitian free resolution of $\hol/\J$ in $D''$
and let $U$ and $R$ be the associated currents, see Section~\ref{URsektion}. 
Dualizing and tensoring with $\Om^p:=\Om^p_M$ we get the complex
\begin{equation}\label{snoe}
0\leftarrow \hol(E_m^*)\otimes\Om^p \stackrel{f_m^*\otimes\text{Id}}{\longleftarrow} \cdots 
\stackrel{f_1^*\otimes\text{Id}}{\longleftarrow} \hol(E_0)\otimes\Om^p \leftarrow 0.
\end{equation}
It is well-known that the cohomology of this complex is isomorphic to $\Ext^{\bullet}(\hol_M/\J,\Om^p)$.
Recall that $\mathscr{C}^{p,q}$ is a stalk-wise injective $\hol_M$-module by a theorem of Malgrange and consider 
the resolution
\begin{equation*}
0\to \Om^p \to \mathscr{C}^{p,0} \stackrel{\debar}{\longrightarrow} \cdots \stackrel{\debar}{\longrightarrow}
\mathscr{C}^{p,N}\to 0
\end{equation*}
of $\Om^p$. Applying the functor $\Hom(\hol_M/\J,-)$ and noticing that 
$\Hom(\hol_M/\J,\mathscr{C}^{p,q})\simeq \mathscr{C}^{p,q}_{\J}$ we get the complex
\begin{equation}\label{snoeis}
0\to \mathscr{C}^{p,0}_{\J} \stackrel{\debar}{\longrightarrow} \cdots \stackrel{\debar}{\longrightarrow}
\mathscr{C}^{p,N}_{\J}\to 0.
\end{equation}
By standard homological algebra, the cohomology of \eqref{snoeis} is naturally isomorphic to the cohomology
of \eqref{snoe}.
Following \cite{MatsNoether} this isomorphism can be realized using the current $R$ as follows.
For convenience we will write just $f^*_k$ instead of $f^*_{k}\otimes \text{Id}$.
If $\xi$ is a holomorphic $p$-form with values in $E^*_q$ such that 
$f^*_{q+1} \xi=0$, then since $fR-\debar R=0$ (see Section~\ref{URsektion})
\begin{equation}\label{lunch}
\debar(\xi\cdot R_q) = \pm \xi\cdot\debar R_q=\pm \xi\cdot f_{q+1}R_{q+1}=\pm f^*_{q+1} \xi\cdot R_{q+1}=0.
\end{equation}
Similarly, if $\xi=f^*_{q} \xi'$, then $\xi\cdot R_q=0$. Thus the map $\hol(E_p^*)\otimes\Om^p\to\mathscr{C}^{p,q}_{\J}$
given by $\xi\mapsto \xi\cdot R_q$ induces a map on cohomology and it turns out to be the natural isomorphism.
Notice that $\xi\cdot R_q$ is a section of $\mathscr{B}_{\J}^{p,q}$.

\begin{proof}[Proof of Theorems~\ref{egenDSthm1} and \ref{egenDSthm2}]
We have already noticed that any $\check{\Proj}$-operator associated to $\J$ maps arbitrary currents to sections of $\mathscr{B}_{\J}$ and that 
any $\check{\K}$-operator associated to $\J$
maps sections of $\mathscr{B}_{\J}$ to sections of $\mathscr{B}_{\J}$.
As $\mathscr{B}_{\J}\subset \PM$ and sections of $\mathscr{B}_{\J}$ have support in $Y$ it follows from the Dimension principle
that $\overline{\sqrt{\J}} \mathscr{B}_{\J}=d \overline{\sqrt{\J}} \wedge \mathscr{B}_{\J}=0$.

Let us temporarily assume that $\J\mathscr{B}_{\J}^{p,q}=0$ and show how 
Theorems~\ref{egenDSthm1} and \ref{egenDSthm2} follow. Then the kernel of $\debar$ in $\mathscr{B}_{\J}^{p,\kappa}$
is $\CH_{\J}^p$ and, by Proposition~\ref{DSweak},
we have $\mu=\debar\check{\K}\mu+\check{\K}\debar\mu+\check{\Proj}\mu$ for any $\mu$ in $\mathscr{B}_{\J}$
if $\check{\Proj}$ and $\check{\K}$ are constructed using the same $R$, $H$, and $g$. Assume henceforth that $g$ depends
holomorphically on $\zeta$.

To show that $\debar$ maps $\mathscr{B}_{\J}$ to itself it suffices to show that if $\mu$ is of the form \eqref{kram} (with $\xi_m=1$)
then $\debar\mu$ is a section of $\mathscr{B}_{\J}$; we will use induction over $m$ to see this. The case $m=0$ follows since
$\nabla_{\text{End}}R=0$. Indeed, then $\mu=\xi\cdot R$ for some smooth $\xi$ and a computation similar to \eqref{lunch}
gives 
\begin{equation*}
\debar(\xi\cdot R)= 
\debar\xi\cdot R\pm f^*\xi\cdot R.
\end{equation*}
If $m>0$ we write $\mu=\check{\K}_m\mu'$, where $\mu'$ is of the form \eqref{kram} with $m$ replaced by $m-1$. 
By Proposition~\ref{DSweak} we get
\begin{equation*}
\debar\mu=\debar\check{\K}_m\mu'=\mu'-\check{\K}_m\debar\mu' - \check{\Proj}_m\mu',
\end{equation*}
and since $\debar\mu'$ is a section of $\mathscr{B}_{\J}$ by the induction hypothesis, it follows that $\debar\mu$ is a
section of $\mathscr{B}_{\J}$.

To see that the inclusion of complexes $(\mathscr{B}_{\J}^{p,\bullet},\debar)\hookrightarrow (\mathscr{C}_{\J}^{p,\bullet},\debar)$
is a quasi-isomorphism, consider the diagram
\begin{equation}\label{diagramm}
\xymatrix{
\mathscr{H}^{q}\big(\hol(E^*_{\bullet})\otimes \Om^p, \,f^*_{\bullet}\big) \ar[r]^{\Psi} \ar[dr] & 
\mathscr{H}^{q}\left(\mathscr{B}_{\J}^{p,\bullet},\,\debar \right) \ar[d] \\
& \mathscr{H}^{q}\left(\mathscr{C}^{p,\bullet}_{\J},\,\debar \right),
}
\end{equation}
where the diagonal map is the natural isomorphism and $\Psi$ is the map $\xi\mapsto \xi\cdot R_q$; cf.\ \cite[Section~7]{MatsNoether}.
It follows that $\Psi$ is injective and that the vertical map is surjective.
Let $\mu$ be a $\debar$-closed section of $\mathscr{B}_{\J}^{p,q}$ and suppose that
there is a $\tau\in\mathscr{C}_{\J}^{p,q-1}$ such that $\mu=\debar\tau$.
Then by Propositions~\ref{DSweak} and \ref{fundprop}
\begin{equation*}
\mu=\debar\check{\K}\mu + \check{\Proj}\mu=\debar(\check{\K}\mu + \check{\Proj}\tau).
\end{equation*}
Since $\check{\K}\mu + \check{\Proj}\tau$ is in $\mathscr{B}_{\J}$ it follows that the vertical map is injective
and so all maps in \eqref{diagramm} are isomorphisms.

Notice that $\debar\check{\Proj}\mu=\check{\Proj}\debar\mu$ for any $\mu\in\mathscr{B}_{\J}$ by Proposition~\ref{fundprop}.
Hence, $\check{\Proj}$ is a map of complexes $(\mathscr{C}_{\J}^{p,\bullet},\debar)\to (\mathscr{B}_{\J}^{p,\bullet},\debar)$.
Showing that it induces the inverse of the vertical map in \eqref{diagramm} is similar to the proof of Proposition~\ref{objective1}:
Let $\mu\in\mathscr{C}_{\J}^{p,q}$ be $\debar$-closed; we must show that $\check{\Proj}\mu$ is a $\debar$-closed section
of $\mathscr{B}_{\J}^{p,q}$ such that $\mu-\check{\Proj}\mu\in \debar\mathscr{C}_{\J}^{p,q-1}$. In view of 
Proposition~\ref{fundprop} the first part of this is clear. The vertical map in
\eqref{diagramm} is an isomorphism and so there is a $\debar$-closed $\nu\in\mathscr{B}_{\J}^{p,q}$ and a $\tau\in \mathscr{C}_{\J}^{p,q-1}$
such that \eqref{grund} holds. Then, by Proposition~\ref{DSweak}, 
$\nu=\check{\Proj}\nu+\debar\check{\K}\nu$ and so by Proposition~\ref{fundprop} we get 
\begin{equation*}
\check{\Proj}\mu=\check{\Proj}\nu+\check{\Proj}\debar\tau=\nu-\debar\check{\K}\nu + \debar\check{\Proj}\tau
\end{equation*}
by applying $\check{\Proj}$ to \eqref{grund}. Subtracting this from \eqref{grund} we get
\begin{equation*}
\mu-\check{\Proj}\mu=\debar(\tau-\check{\Proj}\tau + \check{\K}\nu).
\end{equation*}

To conclude the proof of Theorems~\ref{egenDSthm1} and \ref{egenDSthm2} it remains to show that $\J\mu=0$ for any section $\mu$
of $\mathscr{B}_{\J}^{p,q}$. 
We may assume that $\mu$ is of the form \eqref{kram}. In view of Section~\ref{PMsektion}, the product 
\begin{equation*}
T_{m,q}:=R_{k_m}(z_m)\wedge B_{\ell_m}(z_m,z_{m-1}) \wedge\cdots \wedge B_{\ell_1}(z_1,z_0)\wedge R_q(z_0),\quad
k_j+\ell_j\leq N,
\end{equation*}
is a well-defined pseudomeromorphic current in $D_{z_0}\times \cdots \times D_{z_m}$; here the $R$'s are currents corresponding to 
Hermitian resolutions of $\hol/\J$ and the $B$'s are Bochner-Martinelli type forms as in Section~\ref{intformelsektion}; $R_{k_j}$
is a component of the part of the $j^{\text{th}}$ $R$-current of bidegree $(0,k_j)$ and $B_{\ell_j}$ is the part of 
the $j^{\text{th}}$ Bochner-Martinelli type form of bidegree $(\ell_j,\ell_j-1)$.
In view of 
Definition~\ref{Koperator}, $\mu$ is a sum of push-forwards, under maps $D_{z_0}\times \cdots \times D_{z_m}\to D_{z_m}$, 
of $T_{m,q}$-currents times smooth forms, and therefore it is sufficient to
show that $\J T_{m,q}=0$ where $\J=\J(z_m)$; we will do this by double induction over $m$ and $q$. 
Since $T_{0,q}=R_q$ it is clear that $\J T_{0,q}=0$ for all $q$. Fix now $m>0$ and notice that $T_{m,q}$ has bidegree
$(\sum_{j=1}^m\ell_j, \sum_{j=1}^mk_j+\ell_j - m +q)$ and that $\sum_{j=1}^mk_j+\ell_j - m +q\leq mN-m+q$. 
From the paragraph preceding Definition~\ref{Koperator} it follows that the support of $T_{m,q}$ is contained in
$Y\times\cdots\times Y$ ($m+1$ copies). Moreover, if $z_{s+1}\neq z_s$ for some $0\leq s\leq m-1$,
then $B_{\ell_{s+1}}(z_{s+1},z_s)$ is smooth and so $T_{m,q}$ is a smooth form times the tensor product of two currents
$T_{s,q}$ and $T_{m-s-1,k_{s+1}}$. By induction over $m$ we have $\J T_{m-s-1,k_{s+1}}=0$. It follows that the support
of $\J T_{m,q}$ is contained in $\{z_0=\cdots =z_m\}\cap Y\times\cdots\times Y\simeq Y$, which has codimension
$mN+\kappa$ in $D_{z_0}\times \cdots \times D_{z_m}$. Thus, by the Dimension principle, $\J T_{m,q}=0$ if
$mN-m+q<mN+\kappa$, i.e., if $q<\kappa+m$. In particular, $\J T_{m,\kappa}=0$, which is the induction start
for showing $\J T_{m,q}=0$ by induction over $q$. Let $R=R(z_0)=R_{\kappa}(z_0)+R_{\kappa+1}(z_0)+\cdots$ be the $R$-current
associated to the Hermitian free resolution $(\hol(E_{\bullet}),f_{\bullet})$ as in Section~\ref{URsektion}.
From (the proof of) \cite[Theorem~4.4]{AW1} it follows that outside of the set $Y_k$ where $f_k$ does not have optimal rank,
there is a smooth $\text{Hom}(E_k,E_{k+1})$-valued $(0,1)$-form $\alpha_k=\alpha_k(z_0)$ such that $R_{k}=\alpha_kR_{k-1}$.
Moreover, since $(\hol(E_{\bullet}),f_{\bullet})$ is exact, it follows from the Buchsbaum-Eisenbud criterion that
$\text{codim}\, Y_k\geq k$, $k\geq 1$. Hence, for $z_0\notin Y_{\kappa+1}$ we have $T_{m,\kappa+1}=\alpha_{\kappa+1}T_{m,\kappa}$
and so the support of $\J T_{m,\kappa+1}$ must be contained in the set where $z_0\in Y_{\kappa+1}$. Since it also has support 
contained in $\{z_0=\cdots=z_m\}$ it must in fact have support contained in 
$\{z_0=\cdots =z_m\}\cap Y_{\kappa+1}\times\cdots\times Y_{\kappa+1}\simeq Y_{\kappa+1}$, which has codimension
$\geq mN+\kappa+1$ in $D_{z_0}\times \cdots \times D_{z_m}$. By the Dimension principle, then, $\J T_{m,\kappa+1}=0$.
Continuing in this way get that $\J T_{m,q}=0$ for all $q$.
\end{proof}

The proof shows that the map $\Psi$ of \eqref{diagramm} is an isomorphism. The injectivity followed since
the diagonal map in \eqref{diagramm} is an isomorphism which in turn relies on 
Malgrange's result that $\mathscr{C}^{p,q}$ is stalk-wise injective. However, both surjectivity and injectivity can be showed directly using the 
methods of this paper. To conclude the paper we sketch how this can be done. For the surjectivity of $\Psi$, let 
$\mu$ be a $\debar$-closed section of $\mathscr{B}_{\J}^{p,q}$. Then, by Propositions~\ref{DSweak} and \ref{fundprop},
$\mu=\xi\cdot R^0_q + \debar\check{\K}\mu$, where $f^*_{q+1}\xi=0$, and so the germ of a section of 
$\HH^q(\mathscr{B}_{\J}^{p,\bullet}, \debar)$
defined by $\mu$ is in the image of $\Psi$.

For the injectivity we will use a new kind of weight in our integral formulas to see that the map $\Phi$, 
defined in Proposition~\ref{fundprop}, induces a left inverse of $\Psi$. Notice that $\Phi$ indeed induces a map on cohomology
since $\Phi(\debar\mu)=f_{q+1}^*\Phi(\mu)$ for sections $\mu$ of $\mathscr{B}_{\J}^{p,q}$.
With the setup of Example~\ref{divvikt1}, consider
\begin{eqnarray*}
\check{G}_q^{\epsilon} &:=& \sum_{\ell=0}^q R^{\ell,\epsilon}_q(z)H^{\ell}_q +
\sum_{\ell=0}^{q-1}U^{\ell,\epsilon}_q(z)H^{\ell}_{q-1}f_q +
\sum_{\ell=0}^{q}f_{q+1}(z)U^{\ell,\epsilon}_{q+1}(z)H^{\ell}_q \\
&=:&
R^{\epsilon}_q(z)H_q + U^{\epsilon}_q(z)H_{q-1}f_q + f_{q+1}(z)U^{\epsilon}_{q+1}(z)H_q,
\end{eqnarray*}
which is a smooth section of $\Lambda_{\eta}\otimes \text{Hom}(E_q^{\zeta},E_q^z)$ for any $\epsilon >0$;
notice that $U^{\epsilon}(z)$ and $R^{\epsilon}(z)$ here depend on $z$.
One can check that $\check{G}_q^{\epsilon}$ satisfies the properties of being a weight, with the property
$\check{G}_{q,0}^{\epsilon}(z,z)=1$ construed as $\check{G}_{q,0}^{\epsilon}(z,z)=\text{Id}_{E_q}$.
Let also $g=g(\zeta,z)$ be a suitable weight such that $\zeta\mapsto g(\zeta,z)$ is holomorphic and $z\mapsto g(\zeta,z)$
has compact support, cf.\ Example~\ref{hjalpvikt}. Identifying sections of the $E_j^*$'s with row vectors, 
sections of the $E_j$'s with column vectors, mappings with matrices, and letting $(\cdot)^*$ denote matrix transpose
as in Section~\ref{intopsektion}, we get for any $E_q^*$-valued
holomorphic $p$-form $\xi$, in view of Section~\ref{intformelsektion}, that
\begin{eqnarray*}
\xi^*(\zeta) &=& \int_z\big(\check{G}^{\epsilon}\wedge g\big)_{N,N}^*\wedge\xi^*(z) =
\int_z  \big(R^{\epsilon}_q(z)H_q\wedge g\big)_{N,N}^*\wedge\xi^*(z) \\
&+& \int_z  \big(U^{\epsilon}_q(z)H_{q-1}f_q\wedge g\big)_{N,N}^*\wedge \xi^*(z) + 
\int_z  \big(f_{q+1}(z)U^{\epsilon}_{q+1}(z)H_q \wedge g\big)_{N,N}^*\wedge \xi^*(z).
\end{eqnarray*}
Notice that the last integral vanishes if $f_{q+1}^*\xi^*=0$, and that the second last integral is $f^*_q$-exact.
Since $R=R^0_{\kappa}+R^0_{\kappa+1}+\cdots$, it follows that $R^{\epsilon}_q(z)H_q\to R^0_q(z)H^0_q$ as $\epsilon\to 0$,
and so we see that if $f_{q+1}^*\xi^*=0$, then 
\begin{equation*}
\xi^*(\zeta)= \int_z\big(R^{0}_q(z)H_q^0\wedge g_{N-q}\big)^*\wedge\xi^*(z) + f_q^*\tilde{\xi}^*=
\Phi(\xi R^0_q)^* + f_q^*\tilde{\xi}^*,
\end{equation*} 
where $\tilde{\xi}$ is an $E^*_{q-1}$-valued holomorphic $p$-form. Hence, $\Phi$ induces a left inverse of $\Psi$.

\end{document}